\documentclass[a4paper,12pt]{article}
\usepackage[utf8]{inputenc}
\usepackage[T2A]{fontenc}
\usepackage[english]{babel}

\usepackage{float}

\usepackage[a4paper,margin=1cm,bmargin=1.5cm]{geometry}
\usepackage{parskip}

\usepackage{amsmath}
\usepackage{amssymb}
\usepackage{amsthm}

\usepackage[matrix,arrow,curve]{xy} 

\newtheorem{lemma}{Lemma}
\newtheorem{theorem}{Theorem}
\newtheorem*{remark}{Remark}
\newtheorem*{corollary}{Corollary}

\let\leq\leqslant
\let\geq\geqslant

\theoremstyle{definition}
\newtheorem{definition}{Definition}
    
\usepackage{csquotes}

\usepackage{mathrsfs} 

\usepackage{indentfirst}

\begin{document}

\author{Maria Dmitrieva}
\title{Continuous Deformations of Algebras of Holomorphic Functions on Subvarieties of a Noncommutative Ball}
\date{}

\renewcommand{\abstractname}{\vspace{-\baselineskip}}

\maketitle

\begin{abstract}

We propose a general method for constructing continuous Banach bundles whose fibers are algebras of holomorphic functions on subvarieties of a closed noncommutative ball. These algebras are of the form $\mathcal{A}_d/I_x$, where $\mathcal{A}_d$ is the noncommutative disk algebra introduced by G. Popescu, and $I_x$ is a graded ideal in $\mathcal{ A}_d$, which depends continuously on the point $x$ of the topological space $X$. Similarly, we construct bundles with fibers isomorphic to the algebras $\mathcal{F}_d/I_x$ of holomorphic functions on subvarieties of an open noncommutative ball. Here $\mathcal{F}_d$ is the algebra of free holomorphic functions on the unit ball, which was also introduced by G. Popescu, and $I_x$ is a graded ideal in $\mathcal{F}_d$, which  depends continuously on the point $x$ of the topological space $X$.

\end{abstract}

\section{Introduction}

Consider an open subset $U\subset \mathbb{C}^d$ and its intersection with some algebraic subvariety $\mathcal{V}\subset U$. We will denote the algebra of holomorphic functions on $\mathcal{V}$ by $\mathcal{O}(\mathcal{V})$. Define the subalgebra $\mathcal{A}(\mathcal{V})\subset \mathcal{O}(\mathcal{V})$ of holomorphic functions that can be continuously extended to the closure of $\mathcal{V}$ in $\mathbb {C}^d$. If $U=\mathcal{V}=\mathbb{B}_d$ is an open unit ball in $\mathbb{C}^d$, then the algebra $\mathcal{A}(\mathbb{B}_d)$ is refered to as the ball algebra; it is a classic and well-studied object of multivariate complex analysis (see, for example, \cite{rudin}).

In his paper \cite{Von_Neumann} G.~Popescu defined the so-called noncommutative disk algebra $\mathcal{A}_d$, a free analogue of the algebra $\mathcal{A}(\mathbb{B}_d)$. The algebra $\mathcal{A}_d$ contains the free algebra on $d$ indeterminates as a dense subalgebra. Some time later, in \cite{amazing_stuff}, G.~Popescu defined a free analogue of the algebra $\mathcal{O}(\mathbb{B}_d)$, a locally convex algebra $\mathcal{F}_d$  of free holomorphic functions on the unit ball and also gave an alternative description of the  algebra $\mathcal{A}_d$ in terms of free series.


Several years later G.~Solomon, O.~M.~Shalit and E.~Shamovich in \cite{solomon_shalit_shamovich} defined a noncommutative analog of the algebra $\mathcal{A}(\mathcal{V})$  for an arbitrary homogeneous subvariety of the unit ball. They proved that all such algebras are of the form $\mathcal{A}_d/\overline{I}$, where $I$ is a graded ideal in the free algebra on $d$ indeterminates and $\overline{I}$ denotes its closure in $\mathcal{A}_d$. Noncommutative analogues of algebras $\mathcal{O}(\mathcal{V})$ have been  studied  a lot less and only in some particular cases, see \cite{noncommut_analog_1, noncommut_analog_2, noncommut_analog_3, noncommut_analog_4, noncommut_analog_5, noncommut_analog_6, noncommut_analog_7, noncommut_analog_8}.


In this paper, continuing this line of research, we will consider the quotient algebras $\mathcal{F}_d/\overline{I}$ and interpret them as noncommutative analogs of the algebras $\mathcal{O}(\mathcal{V})$. We will also be interested in continuous families of algebras $\mathcal{F}_d/\overline{I}$ and $\mathcal{A}_d/\overline{I}$ in a situation where the ideal $I$ in some sense continuously depends on a parameter $x\in X$, where $X$ is a topological space. To concretize the phrase continuous family, we will use the notion of a continuous bundle of locally convex (in particular, Banach) algebras.

In \S 3 of this paper, relying on one of the results of O. M. Shalit and B. Solel \cite{subproduct}, we construct continuous bundles of Banach algebras with fibers isomorphic to $\mathcal{A}_d/\overline{I }$. In \S 4 we show that the algebra $\mathcal{F}_d$ is the inverse limit of the algebras $\mathcal{A}_d$. Using this and relying on the construction from \S 3, we construct continuous bundles of locally convex algebras, with fibers isomorphic to $\mathcal{F}_d/\overline{I}$. Thus, we generalize Corollary 8.19 from the paper \cite{pirkovskii} by A.~Yu.~Pirkovskii, which corresponds to the special case when the ideal $I$ is generated by $q$-commutators of generators of the algebra $\mathcal{F}_d$ (in this case $\mathcal{F}_d/\overline{I}$ is the algebra of holomorphic functions on the quantum ball). In \S 5 we will show that our method allows us to construct all possible bundles of some form. In \S 6 we give an alternative definition for the algebra $\mathcal{F}_d/\overline{I}$ in the spirit of the matrix  theory of free noncommutative functions. \cite{vinnikov_book, agler_book}.

\section{Preliminaries}

For any Hilbert space $H$ we denote by $\mathcal{B}(H)$ the Banach algebra of continuous operators on this space.

Take a natural number $d$. We will consider the space $\mathbb{C}^d$ as a Hilbert space with the standard inner product and the standard orthonormal basis $e_1,\ldots, e_d$, where $e_i=(0,\ldots,0,1,0,\ldots,0)$, and one is in the $i$th position. Then for $k=0,1,2,\ldots$ the space $(\mathbb{C}^{d})^{\otimes k}$ is also a Hilbert space, and the set of vectors ${\{e_{i_1}\otimes \ldots \otimes e_{i_{k}}\}_{1\leq i_1,\ldots,i_k\leq d}}$ forms its orthonormal basis.

\begin{definition}
\emph{The full Fock space} $\mathcal{F}(\mathbb{C}^d)$ is the Hilbert direct sum of the spaces $(\mathbb{C}^{d})^{\otimes k}$ (${k =0,1,2,\ldots}$).
\end{definition}

Note that the set $\{e_{i_1}\otimes \ldots \otimes e_{i_{k}}\}_{1\leq i_1,\ldots,i_k\leq d, \, k\geq 0}$ form an orthonormal basis of the space $\mathcal{F}(\mathbb{C}^d)$. \emph{The creation operators} are bounded linear operators $s_j: \mathcal{F}(\mathbb{C}^d)\to \mathcal{F}(\mathbb{C}^d)$ uniquely defined by the rule ${e_ {i_1}\otimes \ldots \otimes e_{i_{k}}\mapsto e_j\otimes e_{i_1}\otimes \ldots \otimes e_{i_{k}}}$. It is easy to see that these operators are isometric.

\begin{definition}
\emph{A noncommutative disk algebra} $\mathcal{A}_d$ is a Banach subalgebra of the algebra $\mathcal{B}(\mathcal{F}(\mathbb{C}^d))$ generated by the operators $s_1,\ldots ,s_d$.

Its dense subalgebra generated (as an algebra) by the operators $s_1,\ldots,s_d$ will be called the \emph{algebra of noncommutative polynomials} and denoted by $\mathcal{P}_d$.
\end{definition}

It is known that all the elements $s_{i_1}\ldots s_{i_k}\in \mathcal{P}_d$ are linearly independent. This implies that the algebra $\mathcal{P}_d$ is isomorphic to the free algebra on $d$ indeterminates (see \cite{Von_Neumann}). Obviously, the algebra $\mathcal{P}_d$ is graded, $\mathcal{P}_d=\bigoplus\limits_{k\geq 0} \mathcal{P}_d^k$, where $\mathcal{P}_d ^k$ is the linear span of elements of the form $s_{i_1}\ldots s_{i_k}$, and the symbol $\oplus$ in this case denotes the direct sum of vector spaces. For each nonzero $p\in \mathcal{P}_d$, we define its degree as a degree of an element of a graded algebra and denote it by $\deg p$.


Consider the free unital semigroup $\mathbb{F}_d^+$ over $d$ generators $g_1,\ldots, g_d$. For each $\alpha=g_{i_1}\ldots g_{i_k}\in \mathbb{F}_d^+$ we denote $|\alpha|:=k$ (in particular, $|1|=0$). Consider the set of formal noncommutative variables $\{z_{1}, \ldots, z_{d}\}$ (that is, generators of a free associative algebra). For each $\alpha=g_{i_1}\ldots g_{i_k}\in \mathbb{F}_d^+$ we set $z_{\alpha}=z_{i_1}\ldots z_{i_k}$.

Suppose $F=\sum\limits_{\alpha\in \mathbb{F}_d^+}a_{\alpha}z_{\alpha}$ is a formal free series with coefficients $a_{\alpha}\in \mathbb {C}$, and $x\in \mathbb{C}$ is an arbitrary number. Then we will denote the formal free series $\sum\limits_{\alpha\in \mathbb{F}_d^+}x^{|\alpha|} a_{\alpha}z_{\alpha}$ by $F^x$ . Also, for any $k\in \mathbb{Z}_{\geq 0}$ we introduce the notation $F_k:=\sum\limits_{\alpha\in \mathbb{F}_d^+:\, |\alpha| =k}a_{\alpha}z_{\alpha}$.

Consider a Hilbert space $H$ and operators $T_1,\ldots,T_d\in \mathcal{B}(H)$. For each $\alpha=g_{i_1}\ldots g_{i_k}\in \mathbb{F}_d^+$ we set $T_{\alpha}=T_{i_1}\ldots T_{i_k}$. If the series $\sum\limits_{k=0}^{\infty}(\sum\limits_{\alpha\in \mathbb{F}_d^+:\, |\alpha|=k} a_{\alpha} T_{\alpha})$ converges in the operator norm in $\mathcal{B}(H)$, then its sum will be denoted by $F(T_1,\ldots,T_d)$. It is obvious that for any $F$ consisting of a finite number of terms and for any ${T_1,\ldots,T_d\in \mathcal{B}(H)}$ the sum $F(T_1,\ldots,T_d)$ is well-defined. It is also easy to see that $$F^x(T_1,\ldots,T_d)=F(xT_1,\ldots,xT_d)$$ and both parts are well-defined simultaneously.

Let $H$ be a Hilbert space.
We define an \emph{open operator ball} of radius $r$ in the space $\mathcal{B}(H)^d$ as follows: $$[\mathcal{B}(H)^d]_r=\{(T_1, \ldots,T_d)\in \mathcal{B}(H)^d:\, \|T_1T_1^*+\ldots+T_dT_d^*\|<r^2\}$$
Similarly, we define $\overline{[\mathcal{B}(H)^d]_r}$, \emph{the closed operator ball} of radius $r$: $$\overline{[\mathcal{B}(H )^d]_r}=\{(T_1,\ldots,T_d)\in \mathcal{B}(H)^d:\, \|T_1T_1^*+\ldots+T_dT_d^*\|\leq r^ 2\}$$

Assume now that $H$ is infinite-dimensional and define the vector space $\mathcal{F}_d^r$ as a subspace in the space of all formal free series with coefficients in $\mathbb{C}$ 
$$\mathcal{F}_d^r=\Bigr\{F=\sum\limits_{\alpha\in \mathbb{F}_d^+}a_{\alpha}z_{\alpha}:\,\forall (T_1,\ldots,T_d)\in [\mathcal{B}(H)^d]_r \,\,\sum\limits_{k=0}^{\infty}\|F_k(T_1,\ldots,T_d)\|<\infty\Bigr\}.$$

We define multiplication on $\mathcal{F}_d^r$ as the multiplication of formal free series, and also a system of seminorms $\|\cdot \|_x$, $0< x<r$, setting $$\|F\|_x: =\sup\limits_{ (T_1,\ldots,T_d)\in [\mathcal{B}(H)^d]_x} \|\sum\limits_{k=0}^{\infty}F_k(T_1, \ldots,T_d)\|.$$

It follows from the results of \cite{amazing_stuff} that $\mathcal{F}_d^r$ is a complete metrizable locally convex algebra, and every seminorm $\|\cdot\|_x$ is actually a norm.

We use the noncommutative von Neumann inequality \cite{Von_Neumann} and see that $$\|F\|_x=\sup\limits_{y<x,
\|T_1T_1^*+\ldots+T_d T_d^*\|=y^2} \|\sum\limits_{k=0}^{\infty}F_k(T_1,\ldots,T_d)\|=\sup \limits_{y<x} \|\sum\limits_{k=0}^{\infty}F_k(ys_1,\ldots,ys_d)\|.$$


The following statement is proved in \cite[Theorem1.5]{amazing_stuff}.
\begin{theorem}
A formal free series $F$ lies in $\mathcal{F}_d^r$ if and only if for any $0< x<r$ the series $\sum\limits_{k=0}^{\infty}x^k\ |F_k(s_1,\ldots,s_d)\|$ converges.
\end{theorem}


Following Popescu's paper \cite{amazing_stuff}, we will interpret the algebra $\mathcal{F}_d^r$ as a free analogue of the algebra of all holomorphic functions on a $d$-dimensional open ball of radius $r$. Unfortunately, this definition is not yet  similar to the definition of the algebra $\mathcal{A}_d$. The connection between these algebras was established by G. Popescu in the same paper \cite{amazing_stuff}, see Theorem  \ref{another_definition_of_A_d} below.

Denote by $\mathcal{A}(\mathcal{B}(\mathcal{X})_r^d)\subset \mathcal{F}_d^r$ the subset of formal free series ${F\in \mathcal{F} _d^r}$ such that the map $${x\mapsto F(xs_1,\ldots,xs_d)}$$ extends continuously from $[0,r)$ to $[0,r]$. Define a norm $\|F\|:=\|F\|_r$ on $\mathcal{A }(\mathcal{B}(\mathcal{X})_r^d)$.

It follows from the results of \cite{amazing_stuff} that $\mathcal{A}(\mathcal{B}(\mathcal{X})_r^d)$ is a Banach algebra. The following result is contained in \cite[Theorem 3.2]{amazing_stuff}.

\begin{theorem} \label{another_definition_of_A_d}
The map $$\Phi: \mathcal{A}(\mathcal{B}(\mathcal{X})_1^d) \to \mathcal{A}_d, \,\, \sum\limits_{\alpha\in \mathbb{F}_d^+} a_{\alpha}z_{\alpha}\mapsto \sum\limits_{k=0}^{\infty} \sum\limits_{\alpha\in \mathbb{F}_d ^+:\, |\alpha|=k} a_{\alpha}s_{\alpha}$$ is an isomorphism of Banach algebras.
\end{theorem}

Theorem \ref{another_definition_of_A_d}, in particular, implies that the elements $z_1,\ldots,z_d\in \mathcal{A}(\mathcal{B}(\mathcal{X})_1^d)$ generate a dense subalgebra $\mathcal{A}(\mathcal{B}(\mathcal{X})_1^d)$.

We now define function algebras on subvarieties of a noncommutative ball.
Below, for the convenience of the reader, we describe some of the constructions that are contained in a much more general form (in the context of subproduct systems) in \cite{subproduct}.

We define a map of vector spaces $Ev: \mathcal{P}_d \to \mathcal{F}(\mathbb{C}^d)$, $s_{i_1}\ldots s_{i_k}\mapsto e_{i_1}\otimes \ldots \otimes e_{i_k}$.
Take an arbitrary two-sided graded ideal $I=\bigoplus\limits_{k\geq 1} I^k\subset \mathcal{P}_d$.

\begin{definition}
We call the subspace $\mathcal{F}_d^I=Ev(I)^{\perp}\subset \mathcal{F}(\mathbb{C}^d)$ \emph{the Hardy space associated with $I$. }
\end{definition}

Denote by $i_I:\mathcal{F}_d^I \hookrightarrow \mathcal{F}(\mathbb{C}^d)$  the canonical embedding and by $p_I: \mathcal{F}(\mathbb{C}^ d) \twoheadrightarrow \mathcal{F}_d^I$ the orthogonal projection. Consider $s^I_j=p_I\circ s_j\circ i_I: \mathcal{F}_d^I\to \mathcal{F}_d^I$.

\begin{definition}
\emph{The noncommutative disk algebra associated with the ideal $I$} is the Banach subalgebra $\mathcal{A}_d^I\subset \mathcal{B}(\mathcal{F}_d^I)$ generated by the operators $s ^I_1,\ldots, s^I_d$.
\end{definition}

Since $\mathcal{A}_d^I$ is an algebra of operators on a Hilbert space, it also has an operator algebra structure in the sense of \cite{big_fat_book}.


O. M. Shalit and B. Solel proved the following remarkable universal property of the algebra $\mathcal{A}_d^I$ \cite[Theorem 8.4]{amazing_stuff}.

\begin{theorem} \label{general_kursovaya}
Let $H$ be a Hilbert space, $T_1,\ldots, T_d\in \mathcal{B}(H)$. Then the following two properties are equivalent.

(i) $(T_1,\ldots, T_d) \in \overline{[\mathcal{B}(H)^d]_1}$ and $p(T_1,\ldots , T_d)=0$.

(ii) There exists a unique completely contractive homomorphism $\phi: \mathcal{A}_d^I\to \mathcal{B}(H)$ such that $\phi(s_i^I)=T_i$.
\end{theorem}

\begin{corollary}
Consider a two-sided graded ideal $I=\bigoplus\limits_{k\geq 1} I^k\subset \mathcal{P}_d$. Then there exists a completely isometric isomorphism $\Phi: \mathcal{A}_d/\overline{I}\to \mathcal{A}_d^I$ such that $\Phi(s_j+\overline{I})=s_j^I $ for all $j=1,\ldots,d$.
\end{corollary}

\begin{proof}
Consider a particular case of the  Theorem \ref{general_kursovaya}: take $I=0$, the space $H=\mathcal{F}(\mathbb{C}^d)$, and the operators $T_k=s_k$ for $k=1,\ldots,d$. We get that ${(s_1,\ldots,s_d)\in \overline{[\mathcal{B}(\mathcal{F}(\mathbb{C}^d))]_1}}$.

Now let $I \subset \mathcal{P}_d$ be an arbitrary two-sided graded ideal. The algebra $\mathcal{A}_d/\overline{I}$ is a quotient algebra of an operator algebra, and therefore is also an operator algebra (see \cite[Theorem 6.3]{big_fat_book}). In other words, there is a completely isometric embedding $\iota: \mathcal{A}_d/\overline{I}\hookrightarrow \mathcal{B}(H)$, $s_i+\overline{I}\mapsto T_i$ for some Hilbert space $H$. In particular, $(T_1,\ldots, T_d) \in \overline{[\mathcal{B}(H)^d]_1}$, because $(s_1,\ldots,s_d)\in \overline{[\ \mathcal{B}(\mathcal{F}(\mathbb{C}^d))]_1}$.
Applying Theorem \ref{general_kursovaya} to the operators $T_1,\ldots, T_k$, we obtain that there exists a completely contractive homomorphism $$\phi_1: \mathcal{A}_d^I\to \operatorname{Im}(\iota)\cong \mathcal{A}_d/\overline{I}$$ such that $\phi_1(s_i^I)=s_i+\overline{I}$.

We apply Theorem \ref{general_kursovaya} to $H=\mathcal{F}_d^I$ and the operators $T_k=s_k^I$ for $k=1,\ldots,d$. We obtain that for each $p\in I$ the condition $p(s_1^I,\ldots,s_d^I)=0$ is satisfied. Now we apply Theorem \ref{general_kursovaya} to $H=\mathcal{F}(\mathbb{C}^d)$ and the operators $T_k=s_k$ for $k=1,\ldots,d$. We obtain that there exists a completely contractive homomorphism $\phi: \mathcal{A}_d\to \mathcal{A}_d^I$ such that $\phi(s_i)=s_i^I$.
Note that $\phi(p)=p(s_1^I,\ldots,s_d^I)=0$ for every $p\in I$. Hence, $\phi(I)=0$, and since $\phi$ is continuous,  $\phi(\overline{I})=0$. Therefore, there exists a completely contractive homomorphism $\phi_2 :\mathcal{A}_d/\overline{I} \to \mathcal{A}_d^I$ such that $\phi_2(s_i+\overline{I})=s_i^ I$.

Since $s_i$ and $s_i^I$ generate dense subalgebras of the corresponding algebras, it follows that $\phi_1\circ \phi_2=\mathbf{1}_{\mathcal{A}_d^I}$ and $\phi_2\circ \phi_1=\mathbf{1}_{\mathcal{A}_d/\overline{I}}$. Hence $\phi_2$ is the required completely isometric isomorphism.
\end{proof}

In what follows, we will identify the algebras $\mathcal{A}_d/\overline{I}$ and $ \mathcal{A}_d^I$ by the isomorphism $\Phi$.

We will use the following definition of a bundle of locally convex algebras, which is close to the definition adopted in \cite{p53} (see also \cite{pirkovskii}).

\begin{definition} \label{definition_bundle}
\emph{A bundle of locally convex algebras} is a quadruple $(E,X,\pi,\mathcal{N})$, where $E$ and $X$ are topological spaces, $\pi:E\to X$ is a continuous open surjection,  for any $x\in X$ on the set $\pi^{-1}(x)$ the structure of an algebra is given, and $\mathcal{N}$ is a directed family of seminorms on $E$. In other words, $\mathcal{N}$ is the set of functions $\sigma: E\to \mathbb{R}_{\geq 0}$ such that $\sigma|_{\pi^{-1 }(x)}$ is a seminorm for any $x$, and for any $ \sigma_1,\sigma_2\in \mathcal{N}$ there exists $ \sigma\in \mathcal{N}$, $C>0 $ such that the conditions $\sigma_1(e)\leq C\sigma(e)$ and $\sigma_2(e)\leq C\sigma(e)$ are satisfied for all $e\in E$.

This quadruple should satisfy the following axioms:

\textbf{A1.} The map $E\times_X E \to E$, $(e_1,e_2)\mapsto e_1+e_2$ is continuous (here $E\times_X E$ denotes the pullback in the category of topological spaces).

\textbf{A2.} The map $E\times \mathbb{C} \to E$, $(e,\lambda)\mapsto \lambda e$ is continuous.

\textbf{A3.} The map $E\times_X E \to E$, $(e_1,e_2)\mapsto e_1e_2$ is continuous.

\textbf{A4.} For a point $x\in X$ denote by $0_x\in E$ the zero vector of the space $\pi^{-1}(x)$. The sets $$\{e\in E: \, \pi(e)\in U, \sigma(e)<\varepsilon\},$$ where $U$ is an open neighborhood of $x$, $\sigma \in \mathcal{N}$ and $\varepsilon>0$, define the local base of $0_x$.

\textbf{A5.} The map $E\to \mathbb{R}$, $e\mapsto \sigma(e)$ is continuous for any $\sigma \in \mathcal{N}$.

Restricting functions from the family $\mathcal{N}$ to the fiber $\pi^{-1}(x)$, we obtain a family of seminorms on $\pi^{-1}(x)$.
It defines a topology which automatically coincides with the topology on $\pi^{-1}(x)$ induced from $E$ (see \cite[Lemma A.17]{pirkovskii}).

\end{definition}

If the spaces $E$, $X$ and the family of seminorms $\mathcal{N}$ satisfy some natural conditions, then, by \cite[Theorem~3.2]{p53}, the definition from \cite{p53} is equivalent to our definition with the axiom  A4 removed. In other words, under these conditions, bundles in our sense are continuous bundles in the sense of \cite{pirkovskii}. On the other hand, in the case when the family $\mathcal{N}$ consists of a single norm and the fibers of the bundle are complete, our definition reduces to the definition of a Banach bundle from \cite{p50}.

Bundles of locally convex algebras admit the following equivalent description in terms of the algebras of their continuous sections. For simplicity, we present it in the particular case of Banach bundles; this will suffice for our purposes.

\begin{definition} \label{definition_banach_2}
\emph{A continuous field of Banach algebras} is a quadruple $(E,X,\pi,\Gamma)$, where $E$ is a set without topology, $X$ is a topological space, $\pi:E\to X$ is a surjection, and for any $x\in X$ on the set $\pi^{-1}(x)$ the structure of a Banach algebra is given. $\Gamma$ is some set of functions $\gamma: X\to E$ such that $\pi\circ \gamma=\operatorname{id}_X$ and the following axioms hold:

\textbf{B1.} For any $\gamma\in \Gamma$ the map $x\mapsto \|\gamma(x)\|$ is continuous.

\textbf{B2.} The set $\Gamma$ is closed under the operation $(\gamma,\delta)\mapsto \gamma+\delta$, where $(\gamma+\delta)(x)=\gamma(x)+\delta (x)$.

\textbf{B3.} For any $\lambda \in \mathbb{C}$ the set $\Gamma$ is closed under the operation $\gamma \mapsto \lambda\gamma$, where ${(\lambda\gamma)(x) =\lambda\gamma(x)}$.

\textbf{B4.} The set $\Gamma$ is closed under the operation $(\gamma,\delta)\mapsto \gamma\delta$, where $(\gamma\delta)(x)=\gamma(x)\delta( x)$.

\textbf{B5.} If the function $\gamma': X \to E$ satisfies the condition $\pi\circ\gamma'=\operatorname{id}_X$, and for any $x\in X$ and $\varepsilon >0$ there exists $\gamma\in \Gamma$ such that $\|\gamma'(x')-\gamma(x')\| < \varepsilon$ for all $x'$ in some neighborhood of $x$, then $\gamma' \in \Gamma$.

\textbf{B6.} For any $x \in X$ the set $\{\gamma(x) | \gamma \in \Gamma \}$ is dense in $\pi^{-1}(x)$.

\end{definition}

As shown in \cite[Theorem 13.18]{p50} (see also \cite[Theorem C.25]{p177}) for every continuous field of Banach algebras $(E,X,\pi,\Gamma)$ there exists a unique topology  on the set $ E$ that turns $(E,X,\pi, \|\cdot \|)$ into a bundle of Banach algebras such that all sections from $\Gamma$ are continuous. And vice versa, if the space $X$ is locally compact or paracompact, then each bundle $(E,X,\pi, \|\cdot \|)$ of Banach algebras $(E,X,\pi,\Gamma)$, where $\Gamma$ is the algebra of all continuous sections of the bundle $E$, corresponds to a continuous field of Banach algebras (see \cite[Appendix C]{p50}). Thus, under the indicated conditions on $X$, the concepts of a bundle of Banach algebras and a continuous field of Banach algebras are essentially equivalent.

For each two-sided graded ideal $I=\bigoplus\limits_{k\geq 1}I^k\subset \mathcal{P}_d$ denote by $p_I$ the orthogonal projection onto the subspace $Ev(I)^{\perp}\subset \mathcal{F}(\mathbb{C}^d)$.

\begin{remark}
It is clear that $\operatorname{Ker}(p_I)=\overline{Ev(I)}=\overline{\bigoplus\limits_{k\geq 1}Ev(I^k)}$. Hence $Ev(I^k)=\operatorname{Ker}(p_I)\cap (\mathbb{C}^d)^{\otimes k}$. Therefore, taking into account the injectivity of the map $Ev$, it can be seen that  the projection $p_I$ uniquely determines all components of $I^k$, and hence the ideal $I$ is also uniquely determined.
\end{remark}

Denote by $\mathcal{M}$ the subset of $\mathcal{B}(\mathcal{F}(\mathbb{C}^d))$ consisting of all such operators $p_I$. We consider $\mathcal{M}$ as a topological space with the topology induced from the strong operator topology on $\mathcal{B}(\mathcal{F}(\mathbb{C}^d))$.

\section{Bundles of function algebras on subvarieties of a closed\\ noncommutative ball}

\begin{theorem}\label{bundle_banach}
Consider a topological space $X$ and a continuous map $\phi: X\to \mathcal{M}$. Then there exists a bundle of Banach algebras $(E,X,\pi,\|\cdot\|)$ such that $\pi^{-1}(x)\cong \mathcal{A}_d/\overline{I_x} $, where for each $x\in X$ the ideal $I_x\subset \mathcal{P}_d$ is uniquely determined by the condition $\phi(x)=p_{I_x}$.
\end{theorem}

\begin{proof}

For any point $x\in X$ we put $\mathcal{A}^{\phi}_d(x)=\mathcal{A}_d^{I_x}=\mathcal{A}_d/\overline{I_x}$ , where $p_{I_x}=\phi(x)$. Let ${\pi_d^\phi(x):\mathcal{A}_d\to \mathcal{A}^{\phi}_d(x)}$ be the canonical projection. We set $E=\bigsqcup\limits_{x\in X}\mathcal{A}^{\phi}_d(x)$ and also define the projection $\pi:E\to X$, ${\pi(\mathcal{A}^{\phi}_d(x))=x}$. We define $\Gamma$ as the set of functions $\gamma: X\to E$ such that $\pi\circ \gamma=\operatorname{id}_X$ and
for any $x_0\in X$, $\varepsilon >0$ there exist $p\in \mathcal{P}_d$ and $U\ni x_0$ such that for all $x\in U$ the following inequality holds: $\| \gamma(x)-\pi_d^\phi(x)( p^\phi)\|_{\mathcal{A}^{\phi}_d(x)}<\varepsilon$. Let us prove that $(E,X,\pi,\Gamma)$ is a bundle in the sense of definition~\ref{definition_banach_2}.

Note that the set of sections $\{\pi_d^\phi(x)( p): \,p\in \mathcal{P}_d\}$ forms an algebra. Therefore, according to \cite[Proposition~10.2.3]{dixmie}, it suffices to prove that the function $x\to \|\pi_d^\phi(x)( p)\|$ is continuous for any $p\in \mathcal{ P}_d$.

We first prove that this function is lower semicontinuous.
Take $x_0\in X$, $\varepsilon>0$. We will assume that $\|\pi_d^\phi(x_0)(p)\|>0$.
We know that $\mathcal{A}^{\phi}_d(x_0)= \mathcal{A}_d^{I_{x_0}}\subset \mathcal{B}(\operatorname{Im}(\phi( x_0)))$. Hence, there exists a vector ${v\in \operatorname{Im}(\phi(x_0))\subset \mathcal{F}(\mathbb{C}^d)}$ such that $\|v\|_{\mathcal {F}(\mathbb{C}^d)}=1$ and $$\|\pi_d^{\phi}(x_0)( p)\|_{\mathcal{A}^{\phi}_d( x_0)}\leq \|\pi_d^{\phi}(x_0)(p)(v)\|_{\mathcal{F}(\mathbb{C}^d)}+\frac{\varepsilon}{ 4}.$$
Since the image of $Ev$ is dense in $\mathcal{F}(\mathbb{C}^d)$, we can assume that $v=Ev(q)$, where $q\in \mathcal{P}_d$ .

Consider $m=\operatorname{deg}(p)+\operatorname{deg}(q)$. Note that since $\operatorname{Ker}(\phi(x))$ is the closure of $Ev(I_x)$ and the ideal $I_x$ is graded, then for any $x\in X$ the operator $\phi (x)$ is graded.
We set $\mathcal{F}^m=\bigoplus\limits_{0\leq t\leq m}(\mathbb{C}^d)^{\otimes t}$ and define the map $\phi_m: X\to \mathcal{B}(\mathcal{F}^m)$ such that  $\phi_m(x)=\phi(x)|_{\mathcal{F}^m}$ for all $x\in X$. Since the space $\mathcal{B}(\mathcal{F}^m)$ is finite-dimensional, the map $\phi_m$ is continuous with respect to the operator norm.

We also define  operators $s_j^m: \mathcal{F}^m \to \mathcal{F}^m,$ $ s_j^m(e_{i_1}\otimes\ldots\otimes e_{i_k})=s_j( e_{i_1}\otimes\ldots\otimes e_{i_k})$ for $k<m$ and $s_j^m(e_{i_1}\otimes\ldots\otimes e_{i_k})=0$ for $k =m$. Obviously, they are continuous.

For all sufficiently small $\varepsilon$ the following inequality holds: $$\frac{\|\pi_d^{\phi}(x_0)( p)\|_{\mathcal{A}^{\phi}_d(x_0)}- \frac{\varepsilon}{2}}{\|\pi_d^{\phi}(x_0)( p)\|_{\mathcal{A}^{\phi }_d(x_0)}-\varepsilon}> 1=\|v\|_{\mathcal{F}(\mathbb{C}^d)}=\|\phi(x_0)(v)\|_{\mathcal{F}(\mathbb{C} ^d)}.$$ Therefore, there exists a neighborhood $U_1\ni x_0$ such that for any $x\in U_1$ $$\|\phi(x)(v)\|_{\mathcal{F }(\mathbb{C}^d)}\leq\frac{\|\pi_d^{\phi}(x_0)( p)\|_{\mathcal{A}^{\phi}_d(x_0)} -\frac{\varepsilon}{2}}{\|\pi_d^{\phi}(x_0)( p)\|_{\mathcal{A}^{\phi }_d(x_0)}-\varepsilon} .$$

Consider $p=\sum\limits_{i=1}^{N} \lambda_{i} s_{k_{1,i}}\ldots s_{k_{n_i,i}}$. The map $$x\mapsto (\sum\limits_{i=1}^{N} \lambda_{i} \phi_m(x)s^m_{k_{1,i}}\phi_m(x)s^m_{ k_{2,i}}\phi_m(x)\ldots \phi_m(x)s^m_{k_{n_i,i}}\phi_m(x))(v)\in \mathcal{F}^m$$ is continuous.
Note that $$(\sum\limits_{i=1}^{N} \lambda_{i} \phi_m(x)s^m_{k_{1,i}}\phi_m(x)s^m_{k_ {2,i}}\phi_m(x)\ldots \phi_m(x)s^m_{k_{n_i,i}}\phi_m(x))(v)=\pi_d^{\phi}(x)( p)( \phi(x)(v)).$$
Therefore, there exists a neighborhood $U_2\ni x_0$ such that for all points $x\in U_2$ the following inequality holds $$\|\pi_d^{\phi}(x)(p)( \phi(x)(v) )-\pi_d^{\phi}(x_0)(p)( \phi(x_0)(v))\|_{\mathcal{F}(\mathbb{C}^d)}=\|\pi_d^ {\phi}(x)(p)( \phi(x)(v))-\pi_d^{\phi}(x_0)(p)( v)\|_{\mathcal{F}(\mathbb{ C}^d)}<\frac{\varepsilon}{4}.$$

Let $U=U_1\cap U_2$. Then for all $ x\in U$ the following inequalities hold$$ \|\pi_d^{\phi}(x_0)(p)(v)\|_{\mathcal{F}(\mathbb{C}^d)} \leq \|\pi_d^{\phi}(x)( p)( \phi(x)(v))\|_{\mathcal{F}(\mathbb{C}^d)}+\frac{ \varepsilon}{4}\leq \|\pi_d^{\phi}(x)( p)\|_{\mathcal{A}^{\phi}_d(x)}\cdot \|\phi(x )(v)\|_{\mathcal{F}(\mathbb{C}^d)}+\frac{\varepsilon}{4};$$
therefore,
$$\|\pi_d^{\phi}(x_0)( p)\|_{\mathcal{A}^{\phi}_d(x_0)}\leq \|\pi_d^{\phi}(x_0) (p)(v)\|_{\mathcal{F}(\mathbb{C}^d)}+\frac{\varepsilon}{4}\leq \|\pi_d^{\phi}(x)( p)\|_{\mathcal{A}^{\phi}_d(x)}\cdot \|\phi(x)(v)\|_{\mathcal{F}(\mathbb{C}^d )}+\frac{\varepsilon}{2}.$$
Hence$$\|\pi_d^{\phi}(x)( p)\|_{\mathcal{A}^{\phi}_d(x)}\geq \frac{\|\pi_d^{\phi}(x_0)( p)\|_{\mathcal{A}^{\phi}_d(x_0)}-\frac{\varepsilon}{2}}{\|\phi(x)(v)\|_{\mathcal{F}(\mathbb{C}^d)}}\geq \|\pi_d^{\phi}(x_0)( p)\|_{\mathcal{A}^{\phi}_d(x_0)}-\varepsilon.$$
Thus the function $x\mapsto \|\pi_d^{\phi}(x)( p)\|$ is lower semicontinuous.

Let us prove that this function is upper semicontinuous.
It follows from the definition of $\mathcal{A}^{\phi}_d(x_0)$ that there exists $q\in \cap I_{x_0}$ such that $\|\pi_d^{\phi}(x_0)( p)\|_{\mathcal{A}^{\phi}_d(x_0)}\geq \|p+q\|_{\mathcal{A}_d}-\frac{\varepsilon}{2}$. Denote $\deg(q)=k$. There is an embedding of vector spaces
$$i_k: \mathcal{F}^k\hookrightarrow \mathcal{A}_d, \quad e_{i_1}\otimes\ldots\otimes e_{i_t}\mapsto s_{i_1}\circ\ldots\circ s_{ i_t}.$$
Note that its domain is a finite-dimensional (Hilbert) space, which means it is continuous. It is also clear that $Ev\circ i_k$ is the canonical embedding  $\mathcal{F}^k\hookrightarrow \mathcal{F}(\mathbb{C}^d)$. Since $Ev(q)\in \mathcal{F}^k$, we can define a map $\psi: X\to \mathcal{A}_d$, $x\mapsto i_k(\phi(x)Ev( q))$. It is continuous as a composition of continuous maps.

Note that $\psi(x_0)=i_k(\phi(x_0)Ev(q))=i_k(0)=0$. Hence, there exists a neighborhood $U_3\ni x_0$ such that $\|\psi(x)\|<\frac{\varepsilon}{2}$ is satisfied for all $x\in U_3$.
Note that $$Ev(\psi(x)-q)=Ev(i_k(\phi(x)Ev(q)))-Ev(q)=(\phi(x)-1)(Ev(q ))\in \operatorname{Ker}(\phi(x))\cap \mathcal{F}^k.$$
Hence, $Ev(\psi(x)-q)\in Ev(I_x)$, and hence $\psi(x)-q\in I_x$. This implies that for all $x\in U_3$ the following inequalities hold: $$\|\pi_d^{\phi}(x)(p)\|_{\mathcal{A}^{\phi}_d(x)} \leq \|p+q-\psi(x)\|_{\mathcal{A}_d}\leq \|p+q\|_{\mathcal{A}_d}+\frac{\varepsilon}{ 2}\leq \|\pi_d^{\phi}(x_0)( p)\|_{\mathcal{A}^{\phi}_d(x_0)}+\varepsilon.$$
Thus the function $x\mapsto \|\pi_d^{\phi}(x)( p)\|$ is upper semicontinuous. Taking into account the lower semicontinuity of this function proved above, we conclude that it is continuous.
\end{proof}

\section{Bundles of function algebras on subvarieties of an open\\ noncommutative ball}

As an auxiliary result, we first prove a non-commutative analog of the fact that an open ball is the union of closed balls embedded in it. To do this, we introduce some notation.

Take $r>0$. For every $0\leq x<r$ we define $\mathcal{A}_d^x=\mathcal{A}_d$.

Then take $0< x<y<r$. It follows from the  Theorem \ref{general_kursovaya} that there exists a contraction homomorphism $\phi_{xy}: \mathcal{A}_d^{y}=\mathcal{A}_d \to \mathcal{A}_d=\mathcal{A }_d^{x}$ such that $\phi_{xy}(s_i)=\frac{x}{y}s_i$ for all $i=1,\ldots,d$.

Consider a graded ideal $I\subset \mathcal{P}_d$. It is easy to see that the homomorphisms $\phi_{xy}$ ($0<x<y<r$) induce homomorphisms $\widetilde{\phi_{xy}}: \mathcal{A}_d^y/\overline{I} \to \mathcal{A}_d^x/\overline{I}$.
We identify the algebra $\mathcal{P}_d$ with a dense subalgebra in $\mathcal{F}_d^r$ by the  embedding $s_{i}\mapsto z_i$ ($i=1,\ldots,d$). The closure of the ideal $I$ in the algebra $\mathcal{F}_d^r$ will also be denoted by $\overline{I}$, this will not lead to confusion.

\begin{theorem} \label{inverse_limit}
For any graded ideal $I\subset \mathcal{P}_d$ there exists an isomorphism of locally convex algebras $${\underleftarrow\lim(\mathcal{A}_d^x/\overline{I}, \widetilde{\phi_{xy}})_{0< x<r}\cong \mathcal{F}_d^r/\overline{I}}.$$
In particular, there exists an isomorphism of locally convex algebras $\underleftarrow\lim(\mathcal{A}_d^x, \phi_{xy})_{0<x<r}\cong \mathcal{F}_d^r$.
\end{theorem}

\begin{proof}
First, we prove the theorem in the case $I=0$.

Denote by $\Phi: \mathcal{A}(\mathcal{B}(\mathcal{X})_1^d) \to \mathcal{A}_d$ the isomorphism of Banach algebras (see Theorem~\ref{another_definition_of_A_d}) given by the conditions $z_i\mapsto s_i$, $i=1,\ldots,d$.
Denote by $A_x$ the algebra $\mathcal{F}_d^r$ considered as a normed algebra with the norm $\|\cdot \|_x$. Let $\pi_x:\mathcal{F}_d^r\to A_x$ be canonical homomorphisms.
For each $x<r$ we define a map ${\theta^0_x: A_x\to \mathcal{A}^x(\mathcal{B}(\mathcal{X})_1^d)}$, $F\mapsto F^x$. It is clear that $\theta^0_x$ is an isometric homomorphism. Hence $$\theta_x=\Phi\circ \theta^0_x: A_x\to \mathcal{A}_d=\mathcal{A}^x_d$$ is also an isometric homomorphism. Note that ${s_i=\theta_x(\frac{1}{x}z_i)}$ generate a dense subalgebra in $\mathcal{A}^x_d$, and hence the image of the homomorphism $\theta_x$ is dense in $\mathcal {A}^x_d$. Hence, since the algebra $\mathcal{A}^x_d$ is Banach, it is isomorphic to the completion of the algebra $A_x$ with respect to the norm $\|\cdot \|_x$.

Let us introduce the notation $r_n:=r(1-\frac{1}{n})$. It is clear that the sequence $\{r_{n}\}_{n\in \mathbb{N}}$ is increasing. Note that the topology on $\mathcal{F}_d^r$ can be generated by a countable family of seminorms $\{\|\cdot\|_{x}\}_{x=r_n, \, n\in \mathbb{N }}$. We know that $\|F\|_{r_n}\leq \|F\|_ {r_{n+1}}$. Therefore, the canonical homomorphism $\rho^0_n: A_{r_{n+1}}\to A_{r_{n}}$ is continuous. Hence, it extends to a continuous homomorphism $\rho_n: \mathcal{A}_d^{r_{n+1}}\to \mathcal{A}_d^{r_{n}}$ given by $s_i\mapsto \frac{r_{n}}{r_{n+1}}s_i$, $i=1,\ldots,d$.

We use \cite[Theorem 3]{allan} and obtain that the algebra $\mathcal{F}_d^r$ together with the family of homomorphisms ${\{\theta_x\circ \pi_x: \mathcal{F}_d^r\to \mathcal{A}_d^x\}_{0<x<r}}$ is the inverse limit of the system $(\mathcal{A}_d^{r_{n}},\rho_n)_{n\in \mathbb {N}}$ in the category of locally convex algebras. This isomorphism defines the Arens-Michael representation (see \cite{allan}) of the algebra $\mathcal{F}_d^r$.

For natural numbers $m>n$ denote by $\rho_{m,n}:\mathcal{A}_d^{r_{m}}\to \mathcal{A}_d^{r_{n}}$ the continuous homomorphism given by the conditions $s_i\mapsto \frac{r_n}{r_m}s_i$, $i=1,\ldots,d$. It is clear that $\rho_{m,n}=\rho_n\circ \rho_{n+1}\circ \ldots\circ \rho_{m-1}$. Note that the homomorphism $\rho_{m,n}$ coincides with the homomorphism $\phi_{r_{m}, r_n}$. Therefore, there are isomorphisms of locally convex algebras $$\mathcal{F}_d^r\cong \underleftarrow\lim(\mathcal{A}_d^{r_{n}},\rho_n)_{n\in \mathbb{N }}\cong\underleftarrow\lim(\mathcal{A}_d^{{r}_{n}},\rho_{m,n})_{m>n}\cong \underleftarrow\lim(\mathcal{ A}_d^x, \phi_{xy})_{x={r}_{n}, \, n\in \mathbb{N}}.$$
Note that the subset $\{{r}_{n}\}_{n\in \mathbb{N}}\subset (0,r)$ is cofinal. Therefore, the inverse limits of the systems $(\mathcal{A}_d^x, \phi_{xy})_{0< x<r}$ and $(\mathcal{A}_d^x, \phi_{xy})_ {x={r}_{n}, \, n\in \mathbb{N}}$ are isomorphic. Hence, there exists an isomorphism of locally convex algebras ${\underleftarrow\lim(\mathcal{A}_d^x, \phi_{xy})_{0<x<r}\cong \mathcal{F}_d^r}$.

We now prove the theorem for an arbitrary graded ideal $I\subset \mathcal{P}_d$.

Denote by $\pi_n: \mathcal{F}_d^r\to \mathcal{A}_d^{{r}_{n}}$ the canonical projections. Then $\overline{\pi_n(\overline{I})}=\overline{\pi_n(I)}=\overline{I}\subset \mathcal{A}_d^{{r}_{n}}$ is closure of the ideal $I$ with respect to the operator norm.
Denote by ${\widetilde{\rho_n}: A_d^{r_{n+1}}/\overline{I}\to A_d^{r_{n}}/\overline{I}}$ and ${\widetilde {\rho_{m,n}}: A_d^{r_{m}}/\overline{I}\to A_d^{r_{n}}/\overline{I}}$ the homomorphisms induced by $\rho_n$ and $\rho_{m,n}$, respectively.
We use \cite[Theorem 6]{allan} and obtain that there is an isomorphism of locally convex algebras $\mathcal{F}_d^r/\overline{I}\cong \underleftarrow\lim(\mathcal{A}_d^{ r_{n}}/\overline{I},\widetilde{\rho_n})_{n\in \mathbb{N}}$.

Similarly to the case $I=0$, we obtain a chain of isomorphisms
$$\mathcal{F}_d^r/\overline{I}\cong \underleftarrow\lim(\mathcal{A}_d^{r_{n}}/\overline{I},\widetilde{\rho_n}) _{n\in \mathbb{N}}\cong\underleftarrow\lim(\mathcal{A}_d^{{r}_{n}}/\overline{I},\widetilde{\rho_{m,n }})_{m>n}\cong\underleftarrow\lim(\mathcal{A}_d^x/\overline{I}, \widetilde{\phi_{xy}})_{x={r}_{n}, \, n\in \mathbb{N}}.$$ From cofinality of $\{{r}_{n}\}_{n\in \mathbb{N}}\subset (0,r) $ we obtain an isomorphism $$\underleftarrow\lim(\mathcal{A}_d^x/\overline{I}, \widetilde{\phi_{xy}})_{0<x<r}\cong\underleftarrow\lim(\mathcal{A}_d^x/\overline{I}, \widetilde{\phi_{xy}})_{x={r}_{n}, \, n\in \mathbb{N}}\cong  \mathcal{F}_d^r/\overline{I}.$$
\end{proof}

Now we will construct a bundle of locally convex algebras with fibers $\mathcal{F}_d/\overline{I}$, where, as in \S 3, the ideal $I$ will depend on the point $x\in X$ ($ X$ is a topological space). Let us formulate the main result of the section.

\begin{theorem} \label{bundle_lca}

Consider a topological space $X$ and a continuous map ${\phi: X\to \mathcal{M}}$.
Then for any $r>0$ there exists a bundle of locally convex algebras $(E,X,\pi,\mathcal{N})$ such that $\pi^{-1}(x)\cong \mathcal{F}_d^r/\overline{I_x} $, where for each $x\in X$ the ideal $I_x\subset \mathcal{P}_d$ is uniquely determined by the condition $\phi(x)=p_{I_x}$.
\end{theorem}

To prove the theorem, we need the following lemma.
\begin{lemma} \label{base}
Consider a bundle of Banach algebras $(E,X,\pi)$ and a set of sections $S\subset \Gamma(X,E)$ such that for any $x\in X$ the set $\{s( x), \,s\in S\}$ is dense in $\pi^{-1}(x)$.
Then sets of the form $$W(V,s,\varepsilon):=\{e\in E: \, \pi_0(e)\in V, \, \|e-s(\pi(e))\|<\varepsilon \},$$ where $V\subset X$ is an open set, $s\in S$, $\varepsilon>0$, form the base of topology of the space $E$.
\end{lemma}

We omit the proof, since it is entirely contained in the proof of Theorem 13.18 in \cite{p50} (see also \cite[Theorem C.25]{p177}).

By Theorem \ref{bundle_banach}, there exists a bundle of Banach algebras $(E_0,X,\pi_0)$ such that ${\pi_0^{-1}(x)\cong \mathcal{A}_d/\overline{ I_x}}$, where for each $x\in X$ the ideal $I_x\subset \mathcal{P}_d$ is uniquely determined by the condition $\phi(x)=p_{I_x}$.
Moreover, from the density of the subalgebra $\mathcal{P}_d$ in $\mathcal{A}_d$ it follows that the set of elements $\pi_d^{\phi}(x)(p)$, where $p\in \mathcal{P}_d$, is dense in each fiber $\pi_0^{-1}(x)$.
Note also that all sections of the form $x\mapsto \pi_d^\phi(x)(p)$ are continuous by the construction of the bundle $(E_0,X,\pi_0)$ (see the proof of the  Theorem \ref{bundle_banach}).
We define sets $$T(V,p,\varepsilon):=\{e\in E_0: \, \pi_0(e)\in V, \, \|e-\pi_d^{\phi}(\pi_0( e))(p)\|<\varepsilon \}$$ for all open $V\subset X$, $\varepsilon>0$, $p\in \mathcal{P}_d$.
It follows from  Lemma \ref{base} that the sets $T(V,p,\varepsilon)$ form the base of  topology of the space $E_0$.

\begin{proof}[Proof of Theorem \ref{bundle_lca}]
For every $0<t<r$ we put $E_0^t=E_0$. For any $0< t<s<r$ and any $x\in X$ the map $\phi_{ts}: \mathcal{A}_d^s\to \mathcal{A}_d^t$ induces the map $\phi ^x_{ts}: \mathcal{A}_d^s/\overline{I_x}\to \mathcal{A}_d^t/\overline{I_x}$, where the ideal $I_x\subset \mathcal{P}_d $ is uniquely defined by the condition $\phi(x)=p_{I_x}$.
Define the map $\Phi_{ts}: E_0^s\to E_0^t$, $e\mapsto \phi^x_{ts}(e)$, where $x=\pi_0(e)$. Let us prove that $\Phi_{ts}$ is continuous. To do this, check that $T(V,p^{\frac{s}{t}},\varepsilon)\subset \Phi_{ts}^{-1}T(V,p,\varepsilon)$. Indeed, if $e\in T(V,p^{\frac{s}{t}},\varepsilon)$, then $x=\pi_0(e)\in V$. Moreover, $\|e-\pi_d^{\phi}(x)(p^{\frac{s}{t}})\|<\varepsilon$, and since the homomorphism $\phi_{ts}^ x$ is contractive, the following estimates hold: $$\|\Phi_{ts}(e)-\pi_d^{\phi}(x)(p)\|=\|\phi^x_{ts}(e-\pi_d^{\phi}(x)(p^{\frac{s}{t}}))\|\leq \|e-\pi_d^{\phi}(x)(p^{\frac{s}{t}})\|<\varepsilon.$$ Therefore, $\Phi_{ts}(e)\in T(V,p,\varepsilon)$. Thus the map $\Phi_{ts}$ is continuous.

We define the space $E$ as the inverse limit of the system $(E_0^t, \Phi_{ts})_{0< t<r}$ in the category of topological spaces. Denote by $\Theta_t: E\to E_0^t$ the canonical maps. We define a map $\pi: E\to X$, $\pi=\pi_0\circ \Theta_{\frac{r}{2}}$.
Note that the set $E$ is the set of sequences $(e^t)_{0< t<r}$ such that $\Phi_{ts}(e^s)=e^t$ for any $0< t<s<r$. In particular, $\pi_0(e^s)=\pi_0(e^t)=x$ for all $0< t<s<r$. But then the set $E$ can be viewed as the set of sequences $(v^t,x)_{0< t<r}$ such that $x\in X$, $v^t\in \mathcal{A }_d/\overline{I_x}$ and $\phi_{ts}^x(v^s)=v^t$ for any $0< t<s<r$. Therefore, the set $\pi^{-1}(x)$ is the inverse limit of the system $(\mathcal{A}_d^t/\overline{I_x}, \phi_{ts}^x)_{0< t<r}$ isomorphic to $ \mathcal{F}_d^r/\overline{I_x}$ by the  Theorem \ref{inverse_limit}. Using this isomorphism, we define an algebraic structure on $\pi^{-1}(x)$ (at this moment we do not claim that it is compatible with the topology).

Denote by $\alpha_t^x: \pi^{-1}(x)\cong \mathcal{F}_d^r/\overline{I_x}\to \mathcal{A}_d^t/\overline{I_x}$ the canonical projections. We introduce a system of seminorms on $E$. Consider $\|e\|_t:=\|\Theta_t(e)\|$.
It is clear that if $t<s$, then $$\|e\|_t=\|\Theta_t(e)\|=\|\Phi_{ts}\circ \Theta_s(e)\|=\|\phi_{ts}^{\pi(e)}\circ \alpha_s^{\pi(e)}(e)\|\leq \|\alpha_s(e)\|=\|\Theta_s(e)\| =\|e\|_s,$$ therefore it is a directed system. Moreover, it is obvious that these seminorms are continuous.

In order to prove that $E$, $X$, $\pi$ and this system of seminorms form a bundle, it remains for us to verify the validity of axioms A1-A4.

Let us start with axiom A4.
The family of sets $T(V,p,\varepsilon)$ forms the basis of topology of each space $E_0^t$. Then the family of sets $\Theta_t^{-1}T(V,p,\varepsilon)$ forms the base of  topology of the space $E$ (see \cite[Proposition 2.5.5]{gen_top}).
Let $x\in X$.
Since axiom A4 holds for the bundle $(E^t_0,X,\pi_0)$, the family of sets $T(V,0,\varepsilon)$, where $V\ni x$ and $\varepsilon>0$, forms a local base at the point $\Theta_t(0_x)=0_x\in E_0$. Therefore, the family of sets
$$\Theta_t^{-1}T(V,0,\varepsilon)=\{e\in E:\, \pi(e)\in V, \,\|\Theta_t(e)\|<\varepsilon \}=\{e\in E:\, \pi(e)\in V, \,\|e\|_t<\varepsilon \}$$ forms a local base at $0_x\in E$. Thus axiom A4 is valid.

The validity of the axioms A1,A2,A3 is easily deduced from the general properties of the limits of bundles. Indeed, we know that $E\times_{X} E=\underleftarrow\lim(E_0^t\times_{X} E_0^t, \Phi_{ts}\times_{X}\Phi_{ts}) _{0\leq t<r}$, since both the fiber product and the inverse limit are limits of the corresponding diagrams.
Let ${Add: E\times_{X} E\to E}$ and ${Add_0: E_0\times_{X} E_0\to E_0}$ be addition, and ${Mult: E\times_{X} E \to E}$ and ${Mult_0: E_0\times_{X} E_0\to E_0}$ be multiplication. Then the composition of the maps ${\Theta_t\circ Add}={Add_0\circ (\Theta_t\times_X \Theta_t)}$ is continuous for any $t$, and hence $Add$ is also continuous. It is proved similarly that the map $Mult$ and the map $Cn: E\times \mathbb{C}\to E$ of multiplication by a constant are continuous. Thus the triple $(E,X,\pi)$, together with the family of seminorms constructed above, satisfies axioms A1-A5 and therefore forms a bundle of locally convex algebras.
\end{proof}

\section{Universality of $\mathcal{M}$}

In this part of the paper, we prove that the topological space $\mathcal{M}$ is a classifying space for bundles of some type.

\begin{lemma} \label{graded_norm}
Let $I=\bigoplus\limits_{m\geq 1}I^m\subset \mathcal{P}_d$ be a graded ideal, $p\in \mathcal{P}_d^m$. Then $\|p\|_{\mathcal{A}_d/\overline{I}}=\|p\|_{\mathcal{P}_d^m/I^m}$.
\end{lemma}

\begin{proof}
It is obvious that $I^m\subset \overline{I}$, and hence $\|p\|_{\mathcal{A}_d/\overline{I}}\leq \|p\|_{\mathcal{P}_d^m/I^m}$. Consider ${\|p\|_{\mathcal{A}_d/\overline{I}}<c< \|p\|_{\mathcal{P}_d^m/I^m}}$. Then there exist $q\in I^m$, $q'\in \bigoplus\limits_{k\neq m}I^k$ such that $\|p+q+q'\|_{\mathcal{A}_d}<c$. On the other hand, ${\|p+q\|_{\mathcal{A}_d}> c}$. Therefore, there exists $v\in \mathcal{F}(\mathbb{C}^d)$ such that $\|(p+q)(v)\|>c\|v\|$, and so ${\|(p+q)(v)\|^2>c^2\|v\|^2}$. We can assume without loss of generality that $$v=v_1+\ldots+v_n, \,\, {v_k\in (\mathbb{C}^d)^{\otimes k}\subset \mathcal{F}(\mathbb {C}^d)}.$$ Then $\|v\|^2=\|v_1\|^2+\ldots+\|v_n\|^2$ and $\|(p+q)(v )\|^2=\|(p+q)(v_1)\|^2+\ldots+\|(p+q)(v_n)\|^2$. Hence $\|(p+q)(v_k)\|^2>c^2\|v_k\|^2$ for some $k$. Note that $$\|(p+q+q')(v_k)\|^2=\|(p+q)(v_k)\|^2+\|q'(v_k)\|^2\geq\|(p+q)(v_k)\|^2>c^2\|v_k\|^2.$$ Hence $\|(p+q+q')(v_k)\|>c\ |v_k\|$ and $\|p+q+q'\|_{\mathcal{A}_d}>c$. This leads to a contradiction.
\end{proof}

Consider a topological space  $X$ and a map $\phi: X\to \mathcal{M}$, which is not assumed to be continuous.
As above, for each $x\in X$ we denote by $I_x$ the graded ideal in $\mathcal{P}_d$ uniquely determined by the condition $\phi(x)=p_{I_x}$.
Suppose that for each $p\in\mathcal{P}_d$ the map $X\to \mathbb{R}_{\geq 0}$, $x\mapsto \|p+\overline{I_x}\|_{ \mathcal{A}_d/\overline{I_x}}$ is continuous.

For a natural number $m$ and $x\in X$ consider the linear map $${f_x:\mathcal{P}_d^m\to \pi^{-1}(x)\cong \mathcal{A}_d/\overline{I_x}},\,\,\,{s_{i_1}\ldots s_{i_m}\mapsto s_{i_1}\ldots s_{i_m}+\overline{I_x}}.$$ We will denote by $|p|_x:=\|f_x(p)\|$ the seminorm on $\mathcal{P}_d^m$ induced by $f_x$.

\begin{lemma}\label{rank}
Let $x_0\in X$. Then there exists an open neighborhood $U\ni x_0$ such that ${\dim \operatorname{Ker} (f_{x})\leq \dim \operatorname{Ker} (f_{x_0 })}$ for each $x\in U$.
\end{lemma}

\begin{proof}

We choose $p_1,\ldots, p_n\in \mathcal{P}_d^m$ such that $f_{x_0}(p_1),\ldots,f_{x_0}(p_n)$ form a basis of $\operatorname{Im} (f_{x_0})$ and $\|p_i\|_{x_0}=1$ for all $i=1,\ldots,n$. Take $E=\operatorname{span}(p_1,\ldots, p_n)$. All norms on $E$ are equivalent, hence there exists $c>0$ such that for any $a_1,\ldots,a_n\in\mathbb{C}$ $$\|a_1p_1+\ldots+a_np_n\|\geq c (|a_1|+\ldots+|a_n|).$$
Note that $|\cdot |_{x_0}$ is also a norm on $E$ equivalent to $\|\cdot \|$. Denote by ${S=\{e\in E: |e|_{x_0}=1\}}$ the unit sphere with respect to $|\cdot |_{x_0}$. Note that $S$ is compact, which means that we can choose a finite number of vectors $e_1,\ldots,e_l\in S$ such that for any $w\in E$ such that $|w|_{x_0}=1 $ there exists ${i\in\{1,\ldots,l\}}$ such that $\|w-e_i\|<\frac{c}{6}$.

For each $i=1,\ldots,n$ there exists an open neighborhood $U_i\ni x_0$ such that $|p_i|_{x}<2$ for any $x\in U_i$.
For every $j=1,\ldots,l$ there is an open neighborhood $V_j\ni x_0$ such that $|e_j|_{x}>\frac{1}{2 }$.

Consider $U=U_1\cap\ldots\cap U_n\cap V_1\cap\ldots\cap V_l$. Take $x\in U$, $w\in E$, $\|w\|=1$. We know that there exists $i\in\{1,\ldots,l\}$ such that $\|w-e_i\|<\frac{c}{6}$. Let $w=a_1p_1+\ldots+a_np_n$, $e_i=b_1p_1+\ldots+b_np_n$. Then $$|w|_{x}\geq |e_i|_{x}-|e_i-w|_{x}\geq \frac{1}{2}-|e_i-w|_{x}. $$ In this case $$|e_i-w|_{x}\leq |a_1-b_1||p_1|_{x}+\ldots+|a_n-b_n||p_n|_{x}<2(|a_1- b_1|+\ldots+|a_n-b_n|).$$
Also $$2(|a_1-b_1|+\ldots+|a_n-b_n|)\leq \frac{2}{c}\|(a_1-b_1)p_1+\ldots+(a_n-b_n)p_n\|=\frac{ 2}{c}\|w-e_i\|<\frac{1}{3}.$$
Therefore, $$|w|_{x}\geq \frac{1}{2}-|e_i-w|_{x}\geq \frac{1}{6}>0.$$
Hence, for all points $x\in U$ and vectors $0\neq w\in E$ it is true that $|w|_{x}=\|w\|\cdot |\frac{w}{\|w \|}|_x\neq 0$, and hence the vectors $f_x(p_1),\ldots,f_x(p_n)$ are linearly independent in $\mathcal{A}_d/\overline{I_{x}}$.
Then $$\dim \operatorname{Ker} (f_{x})=d^m-\dim \operatorname{Im} (f_{x})\leq d^m-n=d^m-\dim \operatorname{Im } (f_{x_0})=\dim \operatorname{Ker} (f_{x_0}).$$
\end{proof}

We denote by $Gr(k,n)$ the complex Grassmannian: the variety of $k$-dimensional subspaces in $\mathbb{C}^{n}$. The Plücker embedding $Pl: Gr(k,n)\to \mathbb{P}(\Lambda^k(\mathbb{C}^n))$ is defined as follows: if a point $x\in Gr(k,n) $ corresponds to the subspace $V\subset \mathbb{C}^n$, the vectors $e_1,\ldots, e_k$ form the basis of the subspace $V$, then $Pl(x)=\mathbb{P}(e_1\wedge\ldots \wedge e_k)$ (see \cite[\S 10]{shafarevich}). It is known that $Pl$ is a smooth embedding of manifolds.

\begin{lemma} \label{Gr}
Let $x_0\in X$, $\dim \operatorname{Ker} (f_{x_0})=k$. Then there exists an open neighborhood $U\ni x_0$ such that $\dim \operatorname{Ker} (f_{x})=k$ for any $x\in U$. In addition, the map $\psi^m: U\to Gr(k,d^m)$, $x\mapsto \operatorname{Ker} (f_{x})$, where $\operatorname{Ker} (f_{ x})$ is considered as a subspace of $\mathcal{P}_d^m$, is continuous.
\end{lemma}

\begin{proof}
Lemma \ref{rank} implies that there exists an open neighborhood $U_0\ni x_0$ such that for any $x\in U_0$ the inequality $\dim \operatorname{Ker} (f_{x})\leq k$ holds.

Take the vectors $p_1,\ldots, p_k\in \operatorname{Ker} (f_{x_0})\subset \mathcal{P}_d^m$, which form the basis of $\operatorname{Ker} (f_{x_0})$.
Consider $\operatorname{Ker} (f_{x_0})$ as a subspace of $\mathcal{P}_d^m$. Consider the corresponding point ${\operatorname{Ker} (f_{x_0})\in Gr(k,d^m)}$. Consider its arbitrary neighborhood $W\subset Gr(k,d^m)$. We show that there exists an open neighborhood $U\ni x_0$ such that for all $x\in U$ the conditions ${\dim \operatorname{Ker} (f_x)= k}$ and $\operatorname{Ker} (f_ {x})\in W$ hold.

Since the Plücker embedding $Pl: Gr(k,d^m)\to \mathbb{P}(\Lambda^k(\mathcal{P}_d^m))$ is a topological embedding, there exists an open neighborhood $W_1\subset \mathbb{P}(\Lambda^k(\mathcal{P}_d^m))$ such that ${W_1\cap Pl(Gr(k,d^m))\subset Pl(W)}$ and ${Pl(\operatorname{Ker} (f_{x_0}))\in W_1}$. Note that there exists an open neighborhood
$W_2\subset \Lambda^k(\mathcal{P}_d^m)\backslash \{0\}$ of $p_1\wedge\ldots\wedge p_k$ such that ${\mathbb{P}((\mathbb{C}\backslash \{0\})\cdot W_2)\subset W_1}$.

Consider the vector space $(\mathcal{P}_d^m)^k$ as a direct sum of the spaces $\mathcal{P}_d^m$ and equip it with the norm given by the equality $\|(r_1,\ldots,r_k)\ |=\max\limits_{i=1,\ldots,k}\|r_i\|$. The map ${Wd: (\mathcal{P}_d^m)^k\to \Lambda^k(\mathcal{P}_d^m)}$ defined by the condition $(r_1,\ldots,r_k)\mapsto r_1 \wedge \ldots\wedge r_k$ is linear in each argument, and hence continuous. Therefore, there exists $\varepsilon>0$ such that $Wd(r_1,\ldots,r_k)\in W_2$ for any $(r_1,\ldots,r_k)$ such that $\|(r_1,\ldots, r_k)-(p_1,\ldots,p_k)\|<\varepsilon$. In particular, $r_1\wedge \ldots\wedge r_k=Wd(r_1,\ldots,r_k)\neq 0,$ and hence $\dim (\operatorname{span}(r_1,\ldots,r_k)) =k$.

For each $i=1,\ldots,k$ consider an open neighborhood $U_i\ni x_0$ such that $|p_i|_{\pi^{-1}(x)}<\varepsilon$ for each $x\in U_i$. Take $U=U_0\cap U_1\cap \ldots\cap U_k$. Lemma \ref{graded_norm} implies that for each $x\in U$ there exists $q_i\in \mathcal{P}_d^m$ such that $q_i\in I_x^m$ and ${\|q_i- p_i\|<\varepsilon}$. Hence $\|(q_1,\ldots,q_k)-(p_1,\ldots,p_k)\|<\varepsilon$.
Then in particular, $\dim (\operatorname{span}(q_1,\ldots,q_k))=k.$ Note that $q_1,\ldots,q_k\in \operatorname{Ker} (f_{x})$. The fact that $x\in U_0$ implies that $\dim \operatorname{Ker} (f_{x})\leq k$. Therefore, $\dim \operatorname{Ker} (f_{x})=k$. On the other hand, ${\operatorname{Ker} (f_{x})=\operatorname{span}(q_1,\ldots,q_k)\in W}$.
\end{proof}

\begin{lemma}\label{one_to_all}
Suppose that for each $m$ the map $\phi^m: X\to \mathcal{B}((\mathbb{C}^d)^{\otimes m})$ given by the condition $x\mapsto p^m_{ I_x}$ is continuous. Then the map $\phi$ is continuous.
\end{lemma}

\begin{proof}
By the definition of $\mathcal{M}$, we need to show that for each $v\in \mathcal{F}(\mathbb{C}^d)$ the map $x\mapsto \phi(x)(v)$ is continuous. Take $x_0\in X$, $\varepsilon>0$.

Since the subspace $Ev(\mathcal{P}_d)\subset\mathcal{F}(\mathbb{C}^d)$ is dense, there exists $p\in \mathcal{P}_d$ such that $w= Ev(p)$, $\|w-v\|<\frac{\varepsilon}{3}$. Consider $p=p_0+\ldots+p_m$, $p_i\in \mathcal{P}_d^i$. Then for every $i=0,\ldots,m$ there exists an open neighborhood $U_i\ni x_0$ such that for every $x\in U_i$ $$\|\phi^i(x)(Ev(p_i)) -\phi^i(x_0)(Ev(p_i))\|<\frac{\varepsilon}{3\sqrt{m+1}}.$$
Take $U=U_0\cap\ldots\cap U_m$, $x\in U$. We use Lemma \ref{graded_norm} and obtain that $$\|\phi(x)(Ev(p))-\phi(x_0)(Ev(p))\|^2=\sum\limits_{i= 0}^m\|\phi^i(x)(Ev(p_i))-\phi^i(x_0)(Ev(p_i))\|^2<(m+1)\frac{\varepsilon^2 }{9(m+1)}=\left(\frac{\varepsilon}{3}\right)^2.$$ Hence, $$\|\phi(x)(w)-\phi(x_0)(w)\ |=\|\phi(x)(Ev(p))-\phi(x_0)(Ev(p))\|<\frac{\varepsilon}{3}.$$ Therefore, $$\|\phi (x)(v)-\phi(x_0)(v)\|\leq \|\phi(x)(v)-\phi(x)(w)\|+\|\phi(x)(w )-\phi(x_0)(w)\|+\|\phi(x_0)(w)-\phi(x_0)(v)\|<\varepsilon.$$
So, for any $\varepsilon>0$, there exists a neighborhood $U\ni x_0$ such that $\|\phi(x)(v)-\phi(x_0)(v)\|<\varepsilon$ for all $ x\in U$. Hence, the map $x\mapsto \phi(x)(v)$ is continuous.
\end{proof}

\begin{theorem} \label{banach_necessary}
Consider a bundle $(E,X,\pi)$ of Banach algebras such that $\pi^{-1}(x)\cong \mathcal{A}_d/\overline{I_x}$ and a section $S_i: x \mapsto s_i+\overline{I_x}$ is continuous for each $i=1,\ldots,d$. Then the map $\phi: X\to \mathcal{M},$ such that for each $x\in X$ the ideal $I_x$ is uniquely determined by the condition $\phi(x)=p_{I_x}$, is continuous.
\end{theorem}

\begin{proof}
We can assume without loss of generality that $X$ is connected.

Consider the homomorphism $F: \mathcal{P}_d\to \Gamma(X,E)$, $s_i\mapsto S_i$. We know that for each $p\in \mathcal{P}_d$ the map $X\to \mathbb{R}_{\geq 0}$, $x\mapsto \|p\|_{\mathcal{A }_d/\overline{I_x}}=\|F(p)(x)\|$ is continuous.

It follows from Lemma \ref{Gr} and the connectedness of $X$ that the function $\dim \operatorname{Ker} (f_{x})$ on $X$ is constant. Let us denote its value by $k$. Lemma \ref{Gr} implies that the map $\psi^m: X\to Gr(k,d^m)$ given by the condition $x\mapsto \operatorname{Ker} (f_{x})$ is continuous. Therefore, the map ${\phi^m: X\to \mathcal{B}((\mathbb{C}^d)^{\otimes m})}$ given by the condition ${x\mapsto p^m_{\operatorname{Ker} (f_{x})^{\perp}}=p^m_{I_x}}$ is continuous. We use Lemma \ref{one_to_all} and obtain that $\phi: X\to \mathcal{M}$ is continuous.
\end{proof}

\begin{theorem}
Consider a bundle $(E,X,\pi,\mathcal{N})$ of locally convex algebras such that $\pi^{-1}(x)\cong \mathcal{F}_d/\overline{I_x }$ and the section $S_i: x\mapsto s_i+\overline{I_x}$ is continuous for every $i=1,\ldots,d$. Suppose that for some $r\in (0,1)$ the map $\sigma_r: E\to \mathbb{R}_{\geq 0}$ defined on each layer as $\pi^{-1}(x )\ni e \mapsto \|e\|_r$ is continuous. Then the map $\phi: X\to \mathcal{M}$, such that for each $x\in X$ the ideal $I_x$ is uniquely determined by the condition $\phi(x)=p_{I_x}$, is continuous.
\end{theorem}

\begin{proof}
Consider the homomorphism $F: \mathcal{P}_d\to \Gamma(X,B)$, $s_i\mapsto \frac{1}{r}S_i$. We know that for each $p\in \mathcal{P}_d$ the map $X\to \mathbb{R}_{\geq 0}$, $x\mapsto \|p\|_{\mathcal{A }_d/\overline{I_x}}=\sigma_r\circ F(p)(x)$ is continuous.

Similarly to the proof of Theorem \ref{banach_necessary}, we obtain that the map $\phi: X\to \mathcal{M}$, where $\phi(x)=p_{I_x}$ for all $x\in X$, is continuous.
\end{proof}

\section{Noncommutative holomorphic functions from a different point of view}

In this section, we give an interpretation of the algebras of the form $\mathcal{F}_d/\overline{I}$ considered in \S 4 in terms of the matrix theory of functions of several free variables \cite{agler_book, vinnikov_book}.
Here we will rely on the work \cite{solomon_shalit_shamovich} by G.~Solomon, O.~M.~Shalit and E.~Shamovich.

Let us first recall some notions and results from \cite{agler_book,  vinnikov_book, solomon_shalit_shamovich}.
We will denote the algebra of matrices of size $n\times n$ by $M_{n\times n}$. We identify it with the algebra $\mathcal{B}(\mathbb{C}^n)$ and thus introduce an operator norm on $M_{n\times n}$. Consider the topological space $\mathbb{M}^d=\bigsqcup_{n\in \mathbb{N}} \mathcal{B}(\mathbb{C}^n)^d$.

\begin{definition}
A subset $\Omega\subset \mathbb{M}^d$ is called a \emph{noncommutative set} (see \cite{agler_book, vinnikov_book}) if  $A\oplus B\in \Omega$ for any $A,B\in \Omega$, where $A\oplus B=\begin{pmatrix}
  A&0\\
  0&B
\end{pmatrix}$ is the direct sum of matrices.
\end{definition}

Each space $M_{n\times n}^d$ can be equipped with an operator norm: $$\|(X_1,\ldots,X_d)\|=\|X_1X_1^*+\ldots+X_dX_d^*\|.$$
\textit{The open noncommutative $d$-dimensional ball} of radius $r$ is a subset $\mathbb{B}_r^d\subset \mathbb{M}^d$ consisting of elements, whose norm is less than $r$:
$$\mathbb{B}_r^d=\{(X_1,\ldots,X_d)\in \mathbb{M}^d: \|(X_1,\ldots,X_d)\|<r\}.$$ It is easy to see that it is an open noncommutative set.

\begin{definition}
(See  \cite{agler_book}, \cite{vinnikov_book}) \emph{A noncommutative function} on a noncommutative subset $\Omega\subset \mathbb{M}^d$ is a sequence of maps $f=(f_n)_{n\in \mathbb{N}}$, $f_n: M_{n\times n}^d\cap \Omega \to M_{n\times n}$ such that the following two conditions are satisfied:

$\bullet$ For any $A\in M_{n\times n}^d\cap \Omega$, $B\in M_{m\times m}^d\cap \Omega$, the equality $f_{n+m}( A\oplus B)=f_n(A)\oplus f_m(B)$ holds.

$\bullet$ For any point $A\in M_{n\times n}^d\cap \Omega$ and any matrix $S\in GL(n)$ such that $SAS^{-1}\in \Omega$ the equality ${f_n(SAS^{-1})=Sf_n(A)S^{-1}}$ holds.
\end{definition}

Thus, a noncommutative function defines a map $f: \mathbb{M}^d\to \mathbb{M}^1$.
It is clear that noncommutative functions can be pointwise added and multiplied.

Note that for any $r>0$ noncommutative polynomials are noncommutative functions on the ball of radius $r$.
Namely, if $X=(X_1,\ldots,X_d)\in M_{n\times n}^d$, we define a homomorphism $\mathcal{P}_d\to M_{n\times n}$, $ p\mapsto p_n^X$ by conditions $s_i\mapsto X_i$. Then the polynomial $p$ corresponds to a noncommutative function $(p_n)_{n\in \mathbb{N}}$, where the map $p_n: M_{n\times n}^d\cap \mathbb{B}_r^d \to M_{n\times n}$ takes $X$ to $p_n^X$.

\begin{definition}
(See \cite{solomon_shalit_shamovich}) A subset $\mathcal{V}\subset \mathbb{M}^d$ is called a \emph{noncommutative algebraic variety} if $$\mathcal{V}=\mathcal{V}_ {\Omega}^S=\{X\in \Omega: \, \forall p\in S\,\, p(X)=0\}$$ for some open noncommutative set $\Omega\subset \mathbb{ M}^d$ and arbitrary $S\subset \mathcal{P}_d$.

If $S\subset \mathcal{P}_d$ is a homogeneous ideal, then $\mathcal{V}_{\Omega}^S$ is called a \emph{homogeneous} noncommutative algebraic variety.
\end{definition}

\begin{definition} (See \cite{agler_book}, \cite{vinnikov_book})
For an open noncommutative set $\Omega\subset \mathbb{M}^d$, a noncommutative function $f: \Omega\to \mathbb{M}^1$ is called \emph{a noncommutative holomorphic function} if it is locally bounded, that is, for each $X\in \Omega$ there exists an open neighborhood $U\subset \Omega$ such that $\sup\limits_{Y\in U}\|f(Y)\|<\infty$.

In particular, if $\sup\limits_{X\in \Omega}\|f(X)\|<\infty$, then $f$ is a noncommutative holomorphic function.
\end{definition}

\begin{remark}
It is known (see \cite[Theorem 12.17]{agler_book}) that for any $n\in \mathbb{N}$ a noncommutative holomorphic function is holomorphic on $M_{n\times n}^d\cap \Omega$ in the usual sense.
\end{remark}

\begin{definition}
(See \cite{solomon_shalit_shamovich}) Consider a noncommutative algebraic variety $\mathcal{V}\subset \mathbb{M}^d$.
A noncommutative function ${f: \mathcal{V}\to \mathcal{M}^1}$ is \emph{holomorphic on $\mathcal{V}$} if for each $X\in \mathcal{V}$ there exists an open noncommutative neighborhood $X\in U\subset \Omega$ and a noncommutative function $g: U\to \mathbb{M}^1$ such that $f|_{\mathcal{V}\cap U}=g |_{\mathcal{V}\cap U}$ and $\sup\limits_{Y\in U}\|g(Y)\|<\infty$.

The algebra $\mathcal{H}^\infty(\mathcal{V})$ \emph{of bounded holomorphic functions} on $\mathcal{V}$ consists of all noncommutative holomorphic functions ${f: \mathcal{V}\to \mathcal{M}^1}$ such that $\|f\|:=\sup\limits_{X\in \mathcal{V}}\|f(X)\|<\infty$.
\end{definition}

Following \cite{amazing_stuff}, we denote by $\mathcal{H}^\infty(\mathcal{B}(\mathcal{X})_r^d)$ the set of formal free series $F\in \mathcal{F}_d ^r$ such that $\|F\|_r<\infty$.
By \cite[Theorem 3.1]{amazing_stuff}, $\mathcal{H}^\infty(\mathcal{B}(\mathcal{X})_r^d)$ is a subalgebra of $\mathcal{F}_d^ r$ and a Banach algebra with respect to the norm $\|\cdot \|_r$.

\begin{theorem} \label{two_balls}
(See \cite[Theorem 3.1]{solomon_shalit_shamovich}) The map $$\Phi: \mathcal{H}^\infty(\mathcal{B}(\mathcal{X})_1^d)\to \mathcal{H }^\infty(\mathbb{B}_1^d), \,\,\, F\mapsto f,$$ given by $f(A)=F(A)$ for any $A\in \mathbb{ B}_1^d$, is an isometric isomorphism of Banach algebras. Here $f(A)$ is the value of the noncommutative function $f$ on $A\in \mathbb{B}_1^d$, and $F(A)$ is the result of substituting $A$ into the formal free series $F$ (see \S 2).
\end{theorem}

From this theorem it is easy to deduce a similar statement about balls of arbitrary radius. Namely, the following result holds true.

\begin{corollary}
The map $$\Phi_r: \mathcal{H}^\infty(\mathcal{B}(\mathcal{X})_r^d)\to \mathcal{H}^\infty(\mathbb{B}_r^d ), \,\,\, F\mapsto f,$$ given by $f(A)=F(A)$ for any $A\in \mathbb{B}_r^d$, is an isometric isomorphism of Banach algebras.
\end{corollary}

\begin{proof}
Note that $\gamma: \mathcal{H}^\infty(\mathcal{B}(\mathcal{X})_1^d)\to \mathcal{H}^\infty(\mathcal{B}(\mathcal{X})_r^d)$, $F\mapsto F^{\frac{1}{r}}$ is an isometric isomorphism since ${\|F^{\frac{1}{r} }\|_r=\|F\|_1}$. On the other hand, the map $\gamma': \mathcal{H}^\infty(\mathbb{B}_1^d)\to \mathcal{H}^\infty(\mathbb{B}_r^d)$, given by $\gamma'(f)(A)=f(\frac{1}{r}A)$ is also an isometric isomorphism, since $\frac{1}{r}\mathbb{B}_r ^d=\mathbb{B}_1^d$. It remains to note that $\Phi_r=\gamma'\circ \Phi\circ \gamma^{-1}$ is the required isomorphism.
\end{proof}

The definition of uniform continuity for noncommutative functions repeats almost verbatim the corresponding definition for maps of metric spaces.

\begin{definition}
(See \cite{solomon_shalit_shamovich}) A noncommutative function $f: \Omega\to \mathbb{M}^1$ is called \emph{uniformly continuous} if for any $\varepsilon>0$ there exists $\delta>0$ such that $\|f(X)-f(Y)\|<
\varepsilon$ for each $n\in \mathbb{N}$, $X,Y\in \Omega\cap M_{n\times n}^d$, $\|X-Y\|<\delta$.
\end{definition}

As proved in \cite[Corollary 9.2]{solomon_shalit_shamovich}, noncommutative polynomials are uniformly continuous on any ball.

Consider  a noncommutative algebraic variety $\mathcal{V}\subset \mathbb{M}^d$.
Following \cite{solomon_shalit_shamovich}, we denote by $\mathcal{A}(\mathcal{V})$ the subalgebra in $\mathcal{H}^\infty(\mathcal{V})$ consisting of functions that can be  continued uniformly continuously to $\overline{\mathcal{V}}$.

\begin{theorem} \label{noncom_func_A}
(See \cite[Proposition 9.7]{solomon_shalit_shamovich}) For a homogeneous ideal $I\subset\mathcal{P}_d$ and a noncommutative algebraic subvariety $\mathcal{V}=\mathcal{V}_{\mathbb{B} _1^d}^I\subset \mathbb{B}_1^d$
there exists a unique isometric isomorphism of algebras $\Phi: \mathcal{A}_d/\overline{I}\to \mathcal{A}(\mathcal{V})$ such that $\Phi(s_i)(X_1,\ldots ,X_d)=X_i$.
\end{theorem}

In particular, $\mathcal{A}_d\cong \mathcal{A}(\mathcal{B}(\mathcal{X})_1^d)\cong \mathcal{A}(\mathbb{B}_1^d )$

The proof of the next result almost verbatim repeats the proof of the corollary from Theorem \ref{two_balls}.

\begin{corollary}
Consider a homogeneous ideal $I\subset\mathcal{P}_d$. Denote by $\mathcal{V}$ the noncommutative algebraic variety $\{X\in \mathbb{M}^d: \, \forall p\in I\,\, p(X)=0\}$.
Then the map
\[\Phi_r: \mathcal{A}(\mathcal{B}(\mathcal{X})_r^d)/\overline{I}\to \mathcal{A}(\mathcal{V}\cap \mathbb {B}_r^d), \,\, F+\overline{I} \mapsto f,\tag{$\ast$}
\]
given by $f(A)=F(A)$ for any $A\in \mathcal{V}\cap \mathbb{B}_r^d$ is an isometric isomorphism of Banach algebras. Here $f(A)$ is the value of the noncommutative function $f$ on $A\in \mathcal{V}\cap \mathbb{B}_r^d$, and $F(A)$ is the result of the substitution $A$ into a formal free series $F$ (see \S 2).
\end{corollary}

\begin{proof}
We know that $\mathcal{A}_d\cong \mathcal{A}(\mathcal{B}(\mathcal{X})_1^d)$, and hence $\Psi(z_i+\overline{I })(X_1,\ldots,X_d)=X_i$ defines an isometric isomorphism of algebras $\Psi: \mathcal{A}(\mathcal{B}(\mathcal{X})_1^d)/\overline{I}\to \mathcal{A}(\mathcal{V}\cap \mathbb{B}_1^d)$.

Note that $\gamma: \mathcal{A}(\mathcal{B}(\mathcal{X})_1^d)/\overline{I}\to \mathcal{A}(\mathcal{B}(\mathcal{X})_r^d)/\overline{I}$, $F\mapsto F^{\frac{1}{r}}$ is an isometric isomorphism since $\|F^{\frac {1}{r}}\|_r=\|F\|_1$. On the other hand, the map $\gamma': \mathcal{A}(\mathcal{V}\cap \mathbb{B}_1^d)\to \mathcal{A}(\mathcal{V}\cap \mathbb{ B}_r^d)$ given by $\gamma'(f)(A)=f(\frac{1}{r}A)$ is also an isometric isomorphism, since $\frac{1}{ r}(\mathcal{V}\cap\mathbb{B}_r^d)=\mathcal{V}\cap\mathbb{B}_1^d$. It remains to note that $\Phi_r=\gamma'\circ \Phi\circ \gamma^{-1}$ is the required isomorphism.
\end{proof}


The following result follows from \cite[Theorems 5.2 and 5.4]{solomon_shalit_shamovich}.

\begin{theorem}\label{noncom_func_hinf_quot}
For a noncommutative algebraic subvariety $\mathcal{V}\subset \mathbb{B}_1^d$
the restriction map ${\mathcal{H}^\infty(\mathbb{B}_1^d)\to \mathcal{H}^\infty(\mathcal{V})}$ is a coisometry.
\end{theorem}

\begin{lemma}\label{image_of_restriction}
Consider $0<x<y$, and  a homogeneous ideal $I\subset\mathcal{P}_d$. Denote the noncommutative algebraic variety $\{X\in \mathbb{M}^d: \, \forall p\in I\,\, p(X)=0\}$ by $\mathcal{V}$. Then the image of the restriction map $$\mathcal{H}^\infty(\mathcal{V}\cap \mathbb{B}_y^d)\to \mathcal{H}^\infty(\mathcal{V}\cap \mathbb{B}_x^d)$$ lies in $\mathcal{A}(\mathcal{V}\cap \mathbb{B}_x^d)$.
\end{lemma}

\begin{proof}

Consider a function $f\in \mathcal{H}^\infty(\mathcal{V}\cap \mathbb{B}_y^d)$. Theorem \ref{noncom_func_hinf_quot} implies that
$f$ is the restriction of some function $\widetilde{f}\in \mathcal{H}^\infty(\mathbb{B}_y^d)$.

Recall (see \S 2) that the algebra $\mathcal{A}(\mathcal{B}(\mathcal{X})_x^d)$ consists of formal free series ${F\in \mathcal{F} _d^x}$ such that the map ${t\mapsto F(ts_1,\ldots,ts_d)}$ can be extended continuously from $[0,x)$ to $[0,x]$. On the other hand, the algebra $\mathcal{H}^\infty(\mathcal{B}(\mathcal{X})_y^d)$ consists of formal free series $F\in \mathcal{F}_d^r$ such that $$\sup\limits_{0< t<y}\|F(ts_1,\ldots,ts_d)\|<\infty$$ (see \cite[Theorem 3.1]{amazing_stuff}). It is clear that for any $F\in \mathcal{H}^\infty(\mathcal{B}(\mathcal{X})_y^d)$ the map ${t\mapsto F(ts_1,\ldots,ts_d) }$ is continuous on $[0,x]\subset [0,y)$. Hence there is an embedding $\mathcal{H}^\infty(\mathcal{B}(\mathcal{X})_y^d) \hookrightarrow \mathcal{A}(\mathcal{B}(\mathcal{X}) _x^d)$.

Note that the composition $$\mathcal{H}^\infty(\mathcal{B}(\mathcal{X})_y^d) \hookrightarrow \mathcal{A}(\mathcal{B}(\mathcal{X} )_x^d)\cong \mathcal{A}(\mathbb{B}_x^d)\twoheadrightarrow \mathcal{A}(\mathcal{V}\cap \mathbb{B}_x^d)\hookrightarrow \mathcal {H}^\infty(\mathcal{V}\cap \mathbb{B}_x^d)$$ is the same as the restriction map $\mathcal{H}^\infty(\mathcal{B}(\mathcal{X} )_y^d)\to \mathcal{H}^\infty(\mathcal{V}\cap \mathbb{B}_x^d)$. Denote by $g$ the image of $\widetilde{f}$ under this composition. Then $g\in \mathcal{A}(\mathcal{V}\cap \mathbb{B}_x^d)$ is the image of $f$ under the original restriction map $$\mathcal{H}^\infty(\mathcal {V}\cap \mathbb{B}_y^d)\to \mathcal{H}^\infty(\mathcal{V}\cap \mathbb{B}_x^d).$$
\end{proof}

\begin{lemma} \label{noncom_func_limits}
Consider $0< x<y<r$ and  a homogeneous ideal $I\subset\mathcal{P}_d$.
Denote by $${\phi_{xy}: \mathcal{A}(\mathcal{V}\cap \mathbb{B}_y^d)\to \mathcal{A}(\mathcal{V}\cap \mathbb {B}_x^d)},\quad \psi_{xy}: \mathcal{H}^\infty(\mathcal{V}\cap \mathbb{B}_y^d)\to \mathcal{H}^ \infty(\mathcal{V}\cap \mathbb{B}_x^d)$$ restriction maps.
Then $\underleftarrow\lim(\mathcal{A}(\mathcal{V}\cap \mathbb{B}_x^d), \phi_{xy})_{0<x<r}\cong \underleftarrow\lim (\mathcal{H}^\infty(\mathcal{V}\cap \mathbb{B}_x^d), \psi_{xy})_{0<x<r}$.

\end{lemma}

\begin{proof}
Take $A=\underleftarrow\lim(\mathcal{A}(\mathcal{V}\cap \mathbb{B}_x^d), \phi_{xy})_{0<x<r}$, $B =\underleftarrow\lim(\mathcal{H}^\infty(\mathcal{V}\cap \mathbb{B}_x^d), \psi_{xy})_{0<x<r}$. Denote by ${\alpha_x: A\to \mathcal{A}(\mathcal{V}\cap \mathbb{B}_x^d)}$ and $\beta_x: B\to \mathcal{H}^\infty (\mathcal{V}\cap \mathbb{B}_x^d)$ the natural projections.

Denote by $\iota_x: \mathcal{A}(\mathcal{V}\cap \mathbb{B}_x^d)\to \mathcal{H}^\infty(\mathcal{V}\cap \mathbb{B }_x^d)$ the inclusion map. Note that the compositions $\iota_x\circ\alpha_x$ induce a continuous homomorphism $\alpha: A\to B$. On the other hand, from Lemma \ref{image_of_restriction} we get the embeddings $\zeta_x: \mathcal{H}^\infty(\mathcal{V}\cap \mathbb{B}_x^d)\to \mathcal{A}( \mathcal{V}\cap \mathbb{B}_{\frac{x^2}{r}}^d)$. Since $\{\frac{x^2}{r}, x\in (0,r)\}=(0,r)$, the compositions $\zeta_x\circ \beta_x$ induce a continuous homomorphism $\beta: B\to A$. It can be seen that $\alpha$ and $\beta$ are inverse to each other, and hence they define a topological isomorphism of locally convex algebras.
\end{proof}

Our next task is to single out some subalgebra in the algebra of all noncommutative holomorphic functions on a homogeneous noncommutative algebraic variety and show that it is isomorphic to the algebra $\mathcal{F}_d^r/\overline{I}$ considered in \S 4.

Consider $r>0$ and  a homogeneous ideal $I\subset\mathcal{P}_d$. Denote the noncommutative algebraic subvariety $\mathcal{V}_{\mathbb{B}_r^d}^I\subset \mathbb{B}_r^d$ by $\mathcal{V}_r$.
By $\mathcal{F}(\mathcal{V}_r)$ we denote the algebra of noncommutative holomorphic functions on $\mathcal{V}_r$, bounded on $\mathcal{V}_r\cap \mathbb{B}_x^ d$ for all $x<r$. We introduce a system of seminorms on it: for $f\in \mathcal{F}(\mathcal{V}_r)$ and $0<x<r$ we set $|f|_x:=\sup\limits_{X\in \mathcal {V}_r\cap \mathbb{B}_x^d} \|f(X)\|$.

The following theorem is the main result of this section.

\begin{theorem}
Let $I\subset\mathcal{P}_d$ be a homogeneous ideal, $r>0$ and $\mathcal{V}_r=\mathcal{V}_{\mathbb{B}_r^d} ^I$.
Then there exists a unique topological isomorphism of algebras $\Theta: \mathcal{F}_d^r/\overline{I}\to \mathcal{F}(\mathcal{V}_r)$ such that $$\Theta(z_i+\overline{I})(X_1,\ldots,X_d)=X_i.$$
\end{theorem}

\begin{proof}
The uniqueness is obvious from the definition of the algebra $\mathcal{F}_d^r$. As above, for any $x,y\in (0,r)$, $x<y$,
consider the restriction maps $${\phi_{xy}: \mathcal{A}(\mathcal{V}\cap \mathbb{B}_y^d)\to \mathcal{A}(\mathcal{V}\cap \mathbb{B}_x^d)}, \quad {\psi_{xy}: \mathcal{H}^\infty(\mathcal{V}\cap \mathbb{B}_y^d)\to \mathcal{H }^\infty(\mathcal{V}\cap \mathbb{B}_x^d)}.$$

We check that $\mathcal{F}(\mathcal{V}_r)\cong \underleftarrow\lim(\mathcal{H}^\infty(\mathcal{V}_r\cap \mathbb{B}_x^d) , \psi_{xy})_{0<x<r}$.
It is clear that there exist restriction maps $${\mathcal{F}(\mathcal{V}_r)\to \mathcal{H}^\infty(\mathcal{V}_r\cap \mathbb{B}_x^d)}.$$ On the other hand, consider a family of noncommutative holomorphic functions $\{f_x\}_{0<x<r}$, ${f_x\in \mathcal{H}^\infty(\mathcal{V}_r\cap \mathbb{ B}_x^d)}$ such that $\psi_{xy}(f_y)=f_x$ for all $0<x<y<r$.
We can pick a function $f: \mathcal{V}_r\cap \mathbb{B}_r^d\to \mathbb{M}^1$, coinciding with $f_x$ on $\mathcal{V}_r\cap \mathbb{ B}_x^d$ for all $0<x<r$. It will be a noncommutative holomorphic function bounded on each  $\mathcal{V}_r\cap \mathbb{B}_x^d$. Hence $\mathcal{F}(\mathcal{V}_r)\cong \underleftarrow\lim(\mathcal{H}^\infty(\mathcal{V}_r\cap \mathbb{B}_x^d), \psi_{xy})_{0<x<r}$.

As in the Theorem \ref{inverse_limit}, for every $0\leq x<r$ we set $\mathcal{A}_d^x=\mathcal{A}_d$. For $0<x<y<r$ consider the homomorphism $$\widetilde{\phi_{xy}}: \mathcal{A}_d^{y}/\overline{I}=\mathcal{A}_d/\overline{I} \to \mathcal{A}_d/\overline{I}=\mathcal{A}_d^{x}/\overline{I},$$
uniquely determined by the condition $\widetilde{\phi_{xy}}(s_i+\overline{I})=\frac{x}{y}s_i+\overline{I}$ for all $i=1,\ldots,d$.

We have a chain of isometric isomorphisms
$$\mathcal{A}_d^x/\overline{I}\to\mathcal{A}(\mathcal{B}(\mathcal{X})_1^d)/\overline{I}\to \mathcal {A}(\mathcal{B}(\mathcal{X})_r^d)/\overline{I}\xrightarrow{\Phi_r} \mathcal{A}(\mathcal{V}\cap \mathbb{B} _r^d),$$
in which the first isomorphism is generated by the isomorphism $\Phi^{-1}$ from Theorem \ref{another_definition_of_A_d}, the second one acts according to the rule $F\mapsto F^{\frac{1}{r}}$ (see the proof of the corollary from Theorem \ref{noncom_func_A}), and the third one is defined in ($\ast$). Denote the composition of these isomorphisms by $\Psi_r$. It is easy to see that $\Psi_r(s_i+\overline{I})(X_1,\ldots,X_d)=\frac{1}{r}X_i$ for all $i=1,\dots,d$ and $x\in \mathcal{V}\cap \mathbb{B}_r^d$.

Let us check that the following diagram is commutative:
$$\xymatrix{
\mathcal{A}_d^x/\overline{I} \ar[d]_{\Psi_x} & \mathcal{A}_d^y/\overline{I} \ar[l]_{\widetilde{\phi_{xy}}} \ar[d]^{\Psi_y} \\
\mathcal{A}(\mathcal{V}_r\cap \mathbb{B}_x^d) &  \mathcal{A}(\mathcal{V}_r\cap \mathbb{B}_y^d) \ar[l]^{\phi_{xy}} 
}$$ 
For all $i=1,\ldots,d$ the following equalities hold
$$(\Psi_x\circ \widetilde{\phi_{xy}}(s_i+\overline{I}))(X_1,\ldots,X_d)=(\Psi_x(\frac{x}{y}s_i+\overline{ I}))(X_1,\ldots,X_d)=\frac{1}{x}\cdot \frac{x}{y}X_i=\frac{1}{y}X_i.$$
On the other hand, $$(\phi_{xy}\circ \Psi_y(s_i+\overline{I}))(X_1,\ldots,X_d)=(\Psi_y(s_i+\overline{I}))(X_1,\ldots,X_d)=\frac{1}{y}X_i.$$
This proves the commutativity of the diagram. From this and Theorem \ref{inverse_limit} it follows that $$\mathcal{F}_d^r/\overline{I}\cong \underleftarrow\lim(\mathcal{A}_d^x/\overline{I}, \widetilde{\phi_{xy}})_{0< x<r}\cong \underleftarrow\lim(\mathcal{A}(\mathcal{V}_r\cap \mathbb{B}_x^d), \phi_{xy})_{0<x<r}.$$
Therefore, by Lemma \ref{noncom_func_limits}, $$\mathcal{F}_d^r/\overline{I}\cong \underleftarrow\lim(\mathcal{A}(\mathcal{V}_r\cap \mathbb{ B}_x^d), \phi_{xy})_{0<x<r}\cong \underleftarrow\lim(\mathcal{H}^\infty(\mathcal{V}_r\cap \mathbb{B} _x^d), \psi_{xy})_{0<x<r}\cong \mathcal{F}(\mathcal{V}_r).$$
If we denote the composition of these maps by $\Theta$, then we see that $\Theta(z_i+\overline{I})(X_1,\ldots,X_d)=X_i$ for all $i=1,\ldots,d$.
\end{proof}


\end{document}